\documentclass[11pt]{article}
\usepackage{amssymb}
\usepackage{latexsym}
\usepackage{amsmath}
\usepackage{amsthm}
\usepackage{amsfonts}
\usepackage{color}
\usepackage{pdfsync}
\usepackage[usenames,dvipsnames]{pstricks}
\usepackage{epsfig}

\newcommand{\fim}{\hfill\rule{2mm}{2mm}}

\numberwithin{equation}{section}

\usepackage{amssymb}

\newcommand{\beq}{\begin{equation} }
\newcommand{\eqq}{\end{equation} }
\newcommand{\cuad}{{\sqcap\kern-.68em\sqcup}}

\newtheorem{teo}{Theorem}[section]

\newtheorem{lemma}{Lemma}[section]
\newtheorem{corollary}{Corollary}[section]
\newtheorem{remark}{Remark}[section]
\newcommand{\bremark}{\begin{remark} \em}
\newcommand{\eremark}{\end{remark} }
\newtheorem{claim}{Claim}

\def\beeq{\begin{equation}}
\def\eeq{\end{equation}}
\newcommand{\begeqaet}{\begin{eqnarray*}}
\newcommand{\eneqaet}{\end{eqnarray*}}

\let\Section=\section
\def\section{\setcounter{equation}{0}\Section}
\newtheorem{Lem}{Lemma}[section]
\newtheorem{Thm}{Theorem}[section]

\newtheorem{Property}{Property}[section]

\begin{document}
\begin{center}{\bf\Large Existence and multiplicity of solutions for a nonlinear Schr\"odinger equation with non-local regional diffusion  }\medskip

\bigskip

\bigskip

{Claudianor O. Alves}

Universidade Federal de Campina Grande\\
Unidade Acad\^emica de Matem\'atica\\
CEP: 58429-900 - Campina Grande - PB, Brazil\\
{\sl coalves@mat.ufcg.edu.br}

\noindent

{C\'esar E. Torres Ledesma}

Departamento de Matem\'aticas, \\
Universidad Nacional de Trujillo,\\
Av. Juan Pablo II s/n. Trujillo-Per\'u\\
{\sl  ctl\_576@yahoo.es}


\medskip

\medskip
\medskip
\medskip
\medskip

\end{center}

\centerline{\bf Abstract}

\medskip

In this article we are interested in  the following non-linear Schr\"odinger equation with non-local regional diffusion 
$$
(-\Delta)_{\rho_\epsilon}^{\alpha}u + u = f(u)   \mbox{ in } \mathbb{R}^{n}, \quad 
u \in H^{\alpha}(\mathbb{R}^{n})\nonumber, \eqno{(P_\epsilon)}
$$
where $\epsilon >0$, $0< \alpha < 1$, $(-\Delta)_{\rho_\epsilon}^{\alpha}$ is a variational version of the regional laplacian, whose range of scope is a ball with radius $\rho_\epsilon(x)=\rho(\epsilon x)>0$, where $\rho$ is a  continuous function. We give general conditions on $\rho$ and $f$  which assure the existence and multiplicity of solution for $(P_\epsilon)$.


\medskip

\date{}

\setcounter{equation}{0}
\section{ Introduction}

The aim of this article is to study the non-linear  Schr\"odinger equation with non-local regional difussion
$$
(-\Delta)_{\rho_\epsilon}^{\alpha}u + u = f(u) \quad \mbox{in}\quad \mathbb{R}^{n},\quad u \in H^{\alpha}(\mathbb{R}^{n}), \eqno{(P_\epsilon)}
$$
where $\epsilon>0$, $0< \alpha <1$, $n\geq 2$ and $f:\mathbb{R} \to \mathbb{R}$ is a $C^1$ function. The operator $(-\Delta)_{\rho_\epsilon}^{\alpha}$ is a variational version of the non-local regional laplacian, with range of scope determined by function $\rho_\epsilon(x)=\rho(\epsilon x)$, where $\rho \in C(\mathbb{R}^{n},(0,+\infty))$. 

As pointed out in \cite{PFCT1}, when studying the singularly perturbed equation (see equation (\ref{Eq04-}) below), the scope function $\rho$, that describes the size of the ball of the influential region of the non-local operator, plays a key role in deciding the concentration point of ground states of the equation. Even though, at a first sight, the minimum point of $\rho$ seems to be the concentration point, there is a non-local effect that needs to be taken in account.

Recently, a great attention has been focused on the study of problems involving the fractional Laplacian, from a pure mathematical point of view as well as from concrete applications, since this operator naturally arises in many different contexts, such as obstacle problems, financial mathematics, phase transitions, anomalous diffusions, crystal dislocations, soft thin films, semipermeable membranes, flame propagations, conservation laws, ultra relativistic limits of quantum mechanics, quasi-geostrophic flows, minimal surfaces, materials science and water waves. The literature is too wide to attempt a reasonable list of references here, so we derive the reader to the work by Di Nezza, Patalluci and Valdinoci \cite{EDNGPEV}, where a more extensive bibliography and an introduction to the subject are given.

In the context of fractional quantum mechanics, non-linear fractional Schr\"odinger equation has been proposed by   Laskin \cite{NL-1}, \cite{NL-2} as a result of expanding the Feynman path integral, from the Brownian-like to the L\'evy-like quantum mechanical paths. In the last 10 years, there has been a lot of interest in the study of the fractional Schr\"odinger equation, see the works in \cite{MC}, \cite{JDMX}, \cite{PFAQJT}, \cite{XGMX} and \cite{EOFCJV}.
In a recent paper Felmer, Quaas and Tan  \cite{PFAQJT} considered positive solutions of nonlinear fractional Schr\"odinger equation
\begin{equation}\label{Eq02-}
(-\Delta)^{\alpha}u + u = f(x,u)\;\;\mbox{in}\;\; \mathbb{R}^{n}.
\end{equation}
They obtained the existence of a ground state by mountain pass argument and a comparison method devised by Rabinowitz in \cite{PR1} for $\alpha = 1$. They  analyzed regularity, decay and symmetry properties of these solutions. At this point it is worth mentioning that the uniqueness of the ground state of $(-\Delta)^\alpha u + u= u^{p+1}$ in $\mathbb{R}$ for general $\alpha \in (0,1)$, where $0<p<4\alpha/(1-2\alpha)$ for $\alpha \in (0,\frac{1}{2})$ and $0< p <\infty$ for $\alpha \in [\frac{1}{2}, 1)$, was proved by Frank and Lenzmann in \cite{RFEL}. Recently, the result of \cite{RFEL} has been extended in any dimension when $\alpha$ is sufficiently close to $1$ by Fall and Valdinoci in \cite{MFEV} and later for general $\alpha \in (0,1)$ by Frank, Lenzmann and Silvestre in \cite{RFELLS}.
We also mention the work by Cheng \cite{MC}, where  the fractional Schr\"odinger equation  
\begin{equation}\label{Eq03-}
(-\Delta)^{\alpha}u + V(x)u = u^{p}\;\;\mbox{in}\;\; \mathbb{R}^{n}
\end{equation}
with unbounded potential $V$ was studied. The existence of a ground state of (\ref{Eq03-}) is obtained by Lagrange multiplier method and the Nehari manifold method is used  to obtain standing waves with prescribed frequency.

On the other hand, research has been done in recent years regarding regional fractional laplacian, where the scope of the operator is restricted to a variable region near each point. We mention the work by Guan \cite{guan1} and Guan and Ma \cite{guan2} where they study these operators, their relation with stochastic processes and they develop integration by parts formula, and the work by Ishii and Nakamura \cite{HIGN}, where the authors studied the Dirichlet problem for regional fractional Laplacian modeled on the $p$-Laplacian. 

Very recently Felmer and Torres \cite{PFCT, PFCT1}, considered positive solutions of nonlinear Schr\"odinger equation with non-local regional diffusion 
\begin{equation}\label{Eq04-}
\epsilon^{2\alpha}(-\Delta)_{\rho}^{\alpha}u + u = f(u)\;\;\mbox{in}\;\;\mathbb{R}^n,\;\;u\in H^{\alpha}(\mathbb{R}^n).
\end{equation}  
The operator $(-\Delta)_{\rho}^{\alpha}$ is a variational version of the non-local regional Laplacian, defined by
$$
\int_{\mathbb{R}^n}(-\Delta)_{\rho}^{\alpha}uvdx = \int_{\mathbb{R}^n}\int_{B(0,\rho (x))} \frac{[u(x+z) - u(x)][v(x+z)-v(x)]}{|z|^{n+2\alpha}}dzdx.
$$ 
Under suitable assumptions on the nonlinearity $f$ and the range of scope $\rho$, they obtained the existence of a ground state by mountain pass argument and a comparison method devised by Rabinowitz in \cite{PR1} for $\alpha = 1$. Furthermore, they analyzed symmetry properties and concentration phenomena of these solutions. These regional operators present various interesting characteristics that make them very attractive from the point of view of mathematical theory of non-local operators.  

Furthermore, in a recent paper \cite{YPJLCT}, Pu, Liu and Tang have considered  the problem
\begin{equation}\label{Eq05-}
(-\Delta)_{\rho}^{\alpha}u + V(x)u = f(u,x) \quad \mbox{in}\quad \mathbb{R}^{n},\quad u \in H^{\alpha}(\mathbb{R}^{n}),
\end{equation}
by assuming that $\rho$ and $V$ are bounded from below and there exist $r_0>0$ such that for any $M>0$,
$$
\lim_{|y|\to \infty}meas(\{x\in \mathbb{R}^n:\;\;|x-y|\leq r_0,\;\;V(x)\leq M\}) = 0,
$$
and the nonlinearity $f(x,u)$ satisfy suitable condition, they have proved the existence of a nonnegative ground state solution for $(\ref{Eq05-})$. Moreover, we should mention that the Dirichlet boundary value problem on a bounded domain with regional diffusion were investigated by the second author in \cite{CT}. 

Motivated by these previous results, in this paper we intend to consider new class of functions $\rho$, more precisely we will consider the following classes: 

\vspace{0.5 cm}

\noindent {\bf Class 1: $\rho$ is periodic}
\begin{enumerate}
\item[($\rho_1$)] $\rho \in C(\mathbb{R}^n,(0,+\infty))$ and  
$$
0<\rho_0=\displaystyle \inf_{x \in \mathbb{R}^n}\rho(x).
$$ 
\item[($\rho_2$)] $\rho (x+T) = \rho (x)$, $x\in \mathbb{R}^n$, $T\in \mathbb{Z}^n$.
\end{enumerate}

\noindent {\bf Class 2:  $\rho$ is asymptotically periodic}
\begin{enumerate}
\item[$(\rho_3)$] There is a continuous periodic function $h_\infty:\mathbb{R}^n \to \mathbb{R}$ such that
$$
0<\rho_0=\inf_{x \in \mathbb{R}^n}\rho(x) \leq \rho(x) \leq h_\infty(x) \quad \forall x \in \mathbb{R}^n. 
$$ 
\item[$(\rho_4)$] 
$$
|h_\infty(x)-\rho(x)|\to 0 \quad \mbox{as} \quad |x| \to +\infty .
$$ 
\end{enumerate}

\noindent {\bf Class 3: $\rho$ has finite global minimum points} \\

The function $\rho$ verifies $(\rho_1)$ and 
\begin{enumerate}
	\item[($\rho_5$)] 
	$$
\rho_\infty=	\lim_{|x| \to +\infty}\rho(x)> \rho(x), \quad \forall x \in \mathbb{R}^n. 
	$$
\item[$(\rho_6)$]  There are only $l$ points $a_1,a_2,\cdots,a_l \in \mathbb{R}^n$ such that 
$$
\rho(a_i)=\displaystyle \inf_{x \in \mathbb{R}^n}\rho(x), \quad \forall i \in \{1,\cdots,l\}.
$$
\end{enumerate}
Without lost of generality, we will assume that 
$$
\displaystyle \inf_{x \in \mathbb{R}^n}\rho(x)=1 \quad \mbox{and} \quad a_1=0.
$$

Associated with the function $f$, we assume the following conditions:  
\begin{enumerate}
\item[($f_1$)] $f\in C^1(\mathbb{R}, \mathbb{R})$ and 
$$
\lim_{|t|\to \infty} \frac{f(t)}{|t|^{q-1}} = 0,\;\;\lim_{|t|\to \infty} \frac{f(t)t}{|t|^2} = +\infty
$$
for some $q\in (2,2_{\alpha}^{*})$, where $2_{\alpha}^{*}=\frac{2n}{n-2\alpha}$.
\item[($f_2$)] $f(t) = o(|t|)$, as $|t| \to 0$.
\item[($f_3$)] There exists $\theta \geq 1$ such that $\theta \mathcal{F}(t) \geq \mathcal{F}(\sigma t)$ for $t\in \mathbb{R}$ and $\sigma \in [0,1]$, where
$$
\mathcal{F}(t) = f(t)t - 2F(t),\;\;\mbox{where} \;\;F(t) = \int_{0}^{t}f(s)ds.
$$
\end{enumerate}

Now we are in a position to state our main existence theorem.
\begin{Thm}\label{main} 
Assume $0 < \alpha < 1$,  $n \geq 2$ and $(f_1)-(f_3)$. If $\epsilon=1$ and  \\
\noindent i) \,\, $\rho$ belongs to Class 1\\ 
or \\
\noindent ii) \,\, $\rho$ belongs to Class 2 and $f$ also satisfies 
$$
\frac{f(t)}{|t|} \quad \mbox{is strictly increasing in} \quad  t, \leqno{(f_4)} 
$$
then problem $(P_\epsilon)$ possesses a non-trivial weak solution.  Moreover, if $\rho$ belongs to Class 3 and $f$ satisfies $(f_1),(f_2),(f_4)$ and 
\begin{enumerate}
\item[($f_3'$)] There exists $\theta >2 $ such that
$$
0<\theta F(t)\leq f(t)t \;\;\mbox{where} \;\;F(t) = \int_{0}^{t}f(s)ds,
$$
\end{enumerate}
then there is $\epsilon_*>0$, such that problem $(P_\epsilon)$ has at least $l$ non-trivial weak solutions for $\epsilon \in (0, \epsilon_*)$.  
\end{Thm}

Before concluding this introduction, we would like point out that in the proof of Theorem \ref{main} we adapt some ideas explored in  Alves, Carri\~ao \& Miyagaki \cite{ACM}, Cao \& Noussair \cite{CN}, Cao \& Zhou \cite{Cao},  Hsu, Lin \& Hu \cite{Hsu1}, Lin \cite{Lin12} and Hu \& Tang \cite{HuTang}.  In the above papers the authors have studied the existence and multiplicity of solution for problems involving the Laplacian operator.

The plan of the paper is as follows: In Section 2, we review some properties of the function space which will be used. In Section 3, we prove some technical lemmas in while in Section 4 we prove the main result. Finally, in Section 5 we write a remark about the existence of ground state solution.

\section{Preliminaries}

The fractional Sobolev space of order $\alpha$ on $\mathbb{R}^n$ is defined by 
$$
H^{\alpha}(\mathbb{R}^n) = \left\{ u\in L^{2}(\mathbb{R}^n):\;\;\int_{\mathbb{R}^n}\int_{\mathbb{R}^n} \frac{|u(x)-u(z)|^2}{|x-z|^{n+2\alpha}}dzdx<\infty \right\},
$$
endowed with the norm
$$
\|u\|_{\alpha} = \left( \int_{\mathbb{R}^n} |u(x)|^2 dx + \int_{\mathbb{R}^n} \int_{\mathbb{R}^n} \frac{|u(x) - u(z)|^2}{|x-z|^{n+2\alpha}}dzdx\right)^{1/2}.
$$
Given a function $\rho$ as above, we define
\begin{equation}\label{2Eq01}
\|u\|^{2} = \int_{\mathbb{R}^n}\int_{B(0,\rho (x))} \frac{|u(x+z)-u(x)|^2}{|z|^{n+2\alpha}}dzdx + \int_{\mathbb{R}^n} |u(x)|^2dx.
\end{equation}
and the space
$$
H_{\rho}^{\alpha}(\mathbb{R}^n) = \{u\in L^2(\mathbb{R}^n):\;\;\|u\|^{2}<\infty\}.
$$
We note that, if $\rho$ satisfies ($\rho_1$), there exists a constant $\tilde{C}>0$ such that
$$
\|u\|_{\alpha} \leq \tilde{C}\|u\|.
$$
This inequality implies that $H_{\rho}^{\alpha}(\mathbb{R}^n) \hookrightarrow L^q(\mathbb{R}^n)$ is continuous for any $q\in [2, 2_{\alpha}^{*}]$ and $H_{\rho}^{\alpha}(\mathbb{R}^n) \hookrightarrow L_{loc}^{q}(\mathbb{R}^n)$ is compact for any $q\in [2, 2_{\alpha}^{*})$ (for more details, see \cite{PFCT1}). From the above remark, we ensure that
$$ 
H^{\alpha}(\mathbb{R}^n)=H_{\rho}^{\alpha}(\mathbb{R}^n)=H_{\rho_\epsilon}^{\alpha}(\mathbb{R}^n)=H_{h_\infty}^{\alpha}(\mathbb{R}^n).
$$
Moreover, the norms $\|\,\,\,\|_\alpha, \, \|\,\,\,\|$ and
$$
\|u\|^{2}_\infty=\int_{\mathbb{R}^n}\int_{B(0,h_\infty (x))} \frac{|u(x+z)-u(x)|^2}{|z|^{n+2\alpha}}dzdx + \int_{\mathbb{R}^n} |u(x)|^2dx.
$$
are equivalents on $H^{\alpha}(\mathbb{R}^n)$.

We would like point out that if $\rho$ is a $\mathbb{Z}^n$-periodic function and $y \in \mathbb{Z}^{n}$, a simple change variable gives
$$
\int_{\mathbb{R}^n}\int_{B(0,\rho(x+y))} \frac{|u(x+z)-u(x)|^2}{|z|^{n+2\alpha}}dzdx = \int_{\mathbb{R}^n}\int_{B(0,\rho (x))} \frac{|u(x+z)-u(x)|^2}{|z|^{n+2\alpha}}dzdx.
$$ 
The above equality will be used frequently in our paper.

The following lemma is a version of the concentration compactness principle proved by Felmer and Torres 
\cite{PFCT1}.
\begin{Lem}\label{lemCC}
Let $n\geq 2$. Assume that $\{u_k\}$ is bounded in $H_{\rho}^{\alpha}(\mathbb{R}^n)$ with
$$
\lim_{k\to \infty} \sup_{y\in \mathbb{R}^n} \int_{B(y,R)} |u_k(x)|^2dx = 0,
$$
for some $R>0$. Then $u_k \to 0$ in $L^q(\mathbb{R}^n)$ for $q\in (2,2_{\alpha}^{*})$.
\end{Lem}

Associated with $(P_\epsilon)$ we have the functional $I: H^{\alpha}_{ \rho_{\epsilon}}(\mathbb{R}^n) \to \mathbb{R}$ defined by
\begin{equation}\label{Eq05}
I(u) = \frac{1}{2}\left( \int_{\mathbb{R}^n}\int_{B(0, \rho_{\epsilon} (x))} \frac{|u(x+z) - u(x)|^2}{|z|^{n+2\alpha}}dzdx + \int_{\mathbb{R}^n}|u(x)|^2dx \right) - \int_{\mathbb{R}^n}F(u(x))dx.
\end{equation}
From $(f_1)$, $I \in C^{1}(H^{\alpha}_{ \rho_{\epsilon}}(\mathbb{R}^n), \mathbb{R})$  with its Fr\'echet derivative given by
$$
\begin{aligned}
I'(u)v &= \int_{\mathbb{R}^n}\int_{B(0,\rho_{\epsilon})}\frac{[u(x+z)-u(x)][v(x+z)-v(x)]}{|z|^{n+2\alpha}} + \int_{\mathbb{R}^n}u(x)v(x)dx\\
& - \int_{\mathbb{R}^n}f(u(x))v(x)dx,
\end{aligned}
$$
for $u,v\in H^\alpha_{ \rho_{\epsilon}}(\mathbb{R}^n)$. Therefore, the critical points of $I$ are weak solutions of $(P_\epsilon)$.

\section{Technical lemmas}

In this section, we are going to prove some technical results, for that purpose we take borrow some ideas of \cite{PFCT1} and \cite{SL}. First all, we would like point out the following properties involving the function $f$:

\begin{Property}\label{nta1}
\begin{enumerate}
\item By condition ($f_1$) and ($f_2$), for any $\tau >0$ there exists a constant $C_\tau >0$ such that
\begin{equation}\label{Eq02}
|F(t)| \leq \tau |t|^2 + C_{\tau}|t|^q.
\end{equation}
\item By  ($f_3$), we deduce that
$$
\mathcal{F}(t) = f(t)t - 2F(t)\geq 0,\;\; \mbox{for all}\;\; t\in \mathbb{R}.
$$
Furthermore, if $t>0$ then we have
\begin{equation}\label{Eq03}
\frac{\partial}{\partial t} \left( \frac{F(t)}{{t^2}} \right) = \frac{tf(t) - 2F(t)}{{ t^3}}\geq 0.
\end{equation}
By ($f_2$),
\begin{equation}\label{Eq04}
\lim_{t\to 0^+} \frac{F(t)}{t^2} = 0.
\end{equation}
Next, from (\ref{Eq03}) and (\ref{Eq04}), we conclude that $F(t)\geq 0$ for all $t\in \mathbb{R}$. 
\end{enumerate}
\end{Property}

Using the above properties we are ready to prove our technical results.

\begin{Lem}\label{mountain}
The functional $I$ satisfies the mountain pass geometry.  
\end{Lem}
\begin{proof}
By (\ref{Eq02}), 
\begin{eqnarray*}
I(u) &\geq& \frac{1}{2} \|u\|^{2} - \tau \|u\|_{L^2}^{2} - C_\tau \|u\|_{L^q}^{q}\\
&\geq& \left( \frac{1}{2} - \tau C_2 \right)\|u\|^{2} - C_\tau C_q\|u\|^{q}.
\end{eqnarray*}
Let $\tau>0$ small enough such that $\frac{1}{2} - \tau C_{2} >0$ and $\|u\| = \zeta$. Since $q>2$, we can take $\zeta$ small enough such that
$$
\frac{1}{2} - \tau C_2 - C_\tau C_q { \zeta }^{q-2} >0.
$$
Therefore 
$$
I(u) \geq \zeta^{2} \left( \frac{1}{2} - \tau C_{2} - C_\tau C_{q} { \zeta}^{q-2} \right) := \beta >0.
$$
\noindent
Now, by ($f_1$), 
$$
\lim_{|t|\to \infty} \frac{F(t)}{|t|^2} = +\infty.
$$
Then, for $\varphi \in C_{0}^{\infty}(\mathbb{R}^n) \setminus \{0\}$, 
$$
\lim_{|t| \to \infty} \int_{\mathbb{R}^n} \frac{F(t\varphi)}{|t|^2}dx = +\infty.
$$ 
Consequently,
$$
\frac{I(t\varphi)}{|t|^2} = \frac{1}{2}\|\varphi\|^{2} - \int_{\mathbb{R}^n} \frac{F(t\varphi)}{|t|^2}dx \to -\infty,\;\;\mbox{as}\;|t| \to \infty.
$$ 
Thereby, setting $t_{0}>0$ large enough and $e = t_{0}\varphi$, we have $I(e) < 0$.
\end{proof}

\begin{Lem}\label{3lem1}
Assume $(f_1)-(f_2)$, $\epsilon =1$ and that $\rho$ belongs to Class 1 or 2.  Let $c\in \mathbb{R}$ and $\{u_k\} \subset H_{\rho}^\alpha(\mathbb{R}^n)$ be a sequence such that
\begin{equation}\label{3Eq00}
I(u_k)\to c \quad \mbox{and} \quad I'(u_k)\to 0\;\;\mbox{as}\;\;k\to \infty.
\end{equation}
Then $\{u_k\}$ is bounded in $H_{\rho}^\alpha(\mathbb{R}^n)$.
\end{Lem}

\begin{proof} To begin with, we recall that
$$
\rho_0 \leq \rho(x) \leq \rho_*, \quad \forall x \in \mathbb{R}^n,
$$	
for $\rho_*=\displaystyle \sup_{x \in \mathbb{R}^n}\rho(x)$. Hence, the functions below 
$$
\|u\|_0=\left( \int_{\mathbb{R}^n}\int_{B(0, \rho_0)} \frac{|u(x+z)-u(x)|^2 }{|z|^{n+2\alpha}}\, dzdx  + \int_{\mathbb{R}^n}|u(x)|^2dx \right)^{\frac{1}{2}} 
$$
and
$$
\|u\|_*=\left( \int_{\mathbb{R}^n}\int_{B(0, \rho_*)} \frac{|u(x+z)-u(x)|^2}{|z|^{n+2\alpha}} \, dzdx  + \int_{\mathbb{R}^n}|u(x)|^2dx \right)^{\frac{1}{2}} 
$$
are equivalents norms to $\|\,\,\,\|$ on $H_{\rho}^\alpha(\mathbb{R}^n)$. Now, arguing by contradiction we suppose that $\{u_k\}$ is unbounded. Then, up to a subsequence, we may assume that
$$
\|u_k\|\to \infty \;\;\mbox{as}\;\;k\to \infty.
$$ 
Thus
\begin{equation}\label{3Eq01}
c = \lim_{k\to \infty}\left( I(u_k) - \frac{1}{2}I'(u_k)u_k \right) = \lim_{k\to \infty} \int_{\mathbb{R}^n} \left(\frac{1}{2}f(u_k(x))u_k(x) - F(u_k(x))  \right)dx.
\end{equation}
Let $w_k = \frac{u_k}{\|u_k\|}$, then $\{w_k\}$ is bounded in $H_{\rho}^\alpha(\mathbb{R}^n)$. We claim that, 
\begin{equation}\label{3Eq02}
\lim_{k\to \infty}\sup_{y\in \mathbb{R}^n} \int_{B(y, 2)} |w_k(x)|^2dx=0.
\end{equation}
Otherwise, for some $\delta >0$, up to a subsequence we have
$$
\sup_{y\in \mathbb{R}^n} \int_{B(y,2)} |w_k(x)|^2dx \geq \delta >0.
$$ 
Let $z_k \in \mathbb{R}^n$ such that
\begin{equation}\label{3Eq03}
\int_{B(z_k ,2)} |w_k(x)|^2dx \geq \tau:= \frac{\delta}{2}>0
\end{equation}
and $v_k (x)= w_k(x+z_k)$. By the change of variable $\tilde{x} = x+ y_k$, we find
$$
\begin{aligned}
\|w_k\|_0 \leq \|v_k\| \leq \|w_k\|_* 
\end{aligned}
$$
from where it follows that  $\{v_k\}$ is also bounded in $H_{\rho}^\alpha(\mathbb{R}^n)$. Passing to a subsequence, we obtain
$$
v_k \to v\;\;\mbox{in}\;\;L_{loc}^{p}(\mathbb{R}^n)\quad \mbox{and} \quad v_k(x) \to v(x)\;\;\mbox{a.e.}\;\;x\in \mathbb{R}^n.
$$
Since
\begin{equation}\label{3Eq04}
\int_{B(0,2)} |v_k(x)|^2dx = \int_{B(z_k, 2)}|w_k(x)|^2dx \geq \tau >0,
\end{equation}
we see that $v\neq 0$. Let $\tilde{u}_k(x) = \|u_k\|v_k(x)$. If $v(x) \neq 0$, we have the limit $|\tilde{u}_k(x)| \to +\infty$ which together with ($f_3$) leads to 
\begin{equation}\label{3Eq05}
\frac{F(\tilde{u}_k(x))}{|\tilde{u}_k(x)|^2}|v_k(x)|^2 \to +\infty.
\end{equation}
The last limit combine with  (\ref{3Eq05}) to give  
\begin{eqnarray}\label{3Eq06}
\frac{1}{2} - \frac{c + o(1)}{\|u_k\|^2} & = & \int_{\mathbb{R}^n} \frac{F(u_k(x))}{\|u_k\|^2}dx\nonumber\\
& = & \int_{\mathbb{R}^n} \frac{F(\tilde{u}_k(x))}{\|u_k\|^2}dx\nonumber\\
&\geq& \int_{\{v\neq 0\}} \frac{F(\tilde{u}_k(x))}{|\tilde{u}_k(x)|}|v_k(x)|^2dx \to +\infty,
\end{eqnarray} 
which is impossible. This shows (\ref{3Eq02}). Then, by Lemma \ref{lemCC} 
\begin{equation}\label{3Eq07}
w_k \to 0\;\;\mbox{in}\;\;L^q(\mathbb{R}^n),\;\;\forall \;q\in (2,2_{\alpha}^{*}).
\end{equation}
We are going to get a contradiction as follow. By { Property \ref{nta1} - (1)}, given $\tau >0$, there exists $C_\tau >0$ such that
\begin{equation}\label{3Eq08}
|F(t)| \leq \tau |t|^2 + C_\tau |t|^q.
\end{equation}
Since $\|w_k\| = 1$, there exists a constant $K>0$ such that
$$
\|w_k\|_{L^2}^{2} \leq K.
$$
Therefore, by (\ref{3Eq07}) and (\ref{3Eq08}) 
$$
\limsup_{k\to \infty} \int_{\mathbb{R}^n} |F(w_k(x))|dx \leq \limsup_{k\to \infty} (\tau \|w_k\|_{L^2}^{2} +C_\tau \|w_k\|_{L^q}^{q}) \leq \epsilon K.
$$
Since $\tau$ is arbitrary, we deduce
\begin{equation}\label{3Eq09}
\lim_{k\to \infty} \int_{\mathbb{R}^n} F(w_k(x))dx = 0.
\end{equation}
Now, we choose a sequence $\{t_k\} \in [0,1]$ such that 
\begin{equation}\label{3Eq10}
I(t_ku_k) = \max_{t\in [0,1]}I(tu_k).
\end{equation}
Given $\sigma >0$, noting that $\frac{(4\sigma)^{1/2}}{\|u_k\|} \in (0,1)$ for $k$ large enough, (\ref{3Eq09}) ensures that   
$$
\begin{aligned}
I(t_ku_k) &\geq I((4\sigma)^{1/2}w_k) = \frac{1}{2}\|(4\sigma)^{1/2}w_k\|^2 - \int_{\mathbb{R}^n}F((4\sigma)^{1/2}w_k(x))dx\\
& = 2\sigma - \int_{\mathbb{R}^n} F((4\sigma)^{1/2}w_k(x))dx \geq \sigma.
\end{aligned}
$$
Namely, $I(t_k u_k) \to +\infty$. But $I(0) = 0$ and $I(u_k) \to c$, then by (\ref{3Eq10}) we see that $t_k\in (0,1)$ and 
\begin{equation}\label{3Eq11}
\begin{aligned}
0 &= t_k \frac{d}{dt}I(tu_k)\Big|_{t=t_k}\\
&= \int_{\mathbb{R}^n} \int_{B(0,\rho (x))} \frac{|t_ku_k(x+z) - t_ku_k(x)|^2}{|z|^{n+2\alpha}}dzdx + \int_{\mathbb{R}^n} V(x)|t_ku_k(x)|^2dx \\
&- \int_{\mathbb{R}^n} f(t_ku_k(x))t_ku_k(x)dx.
\end{aligned}
\end{equation}
Now from (\ref{3Eq11}) and $(f_3)$, 
$$
\begin{aligned}
\int_{\mathbb{R}^n} \left(\frac{1}{2}f(u_k)u_k - F(u_k)  \right) dx &\geq \frac{1}{\theta} \int_{\mathbb{R}^n} \left( \frac{1}{2}f(t_ku_k)t_ku_k - F(t_ku_k) \right)dx\\
& = \frac{1}{\theta} \left( \frac{1}{2}\|t_ku_k\|^2 - \int_{\mathbb{R}^n}F(t_ku_k)dx \right)\\
& = \frac{1}{\theta} I(t_ku_k) \to +\infty.
\end{aligned}
$$
This contradicts with (\ref{3Eq01}). Thereby, $\{u_k\}$ is bounded.
\end{proof}

\section{Proof of Theorem \ref{main}}

In the sequel, we will analysis  the classes $(\rho_1),(\rho_2)$ and $(\rho_3)$ separately. 

\subsection{Class 1: $\rho$ is periodic}

Let $c = \inf_{\gamma \in \Gamma} \max_{t\in [0,1]} I(\gamma (t)) >0$, then by the Ekeland  variational principle, there is a sequence $\{u_k\}$ such that
$$
I(u_k)\to c\;\;\mbox{and}\;\;I'(u_k)\to 0.
$$
By Lemma \ref{3lem1}, $\{u_k\}$ is bounded in $H_{\rho}^\alpha(\mathbb{R}^n)$. In what follows, fix  
\begin{equation}\label{3Eq12}
\delta = \lim_{n\to \infty} \sup_{y\in \mathbb{R}^n} \int_{B(y,2)} |u_k(x)|^2dx.
\end{equation}
If $\delta = 0$, the Lemma \ref{lemCC} yields  
$$
u_k\to 0\;\;\mbox{in}\;\;L^q(\mathbb{R}^n),\;\;\forall \;q\in (2,2_\alpha^{*}).
$$
Then, arguing as in (\ref{3Eq09}),
\begin{equation}\label{3Eq13}
\lim_{k\to \infty}\int_{\mathbb{R}^n} F(u_k(x))dx = 0\;\;\mbox{and}\;\;\lim_{k\to \infty} \int_{\mathbb{R}^n} f(u_k(x))u_k(x)dx = 0.
\end{equation}
The above limits together with (\ref{3Eq01}) implies that $c = 0$, a contradiction. Therefore $\delta >0$. So there exists a sequence $\{y_k\} \subset \mathbb{Z}^n$ and a real number $\tau>0$ such that
\begin{equation}\label{3Eq14}
\int_{B(0, 2)} |v_k(x)|^2dx = \int_{B(y_k, 2)}|u_k(x)|^2dx > \tau,
\end{equation}
where $v_k(x) = u_k(x+y_k)$. Moreover, since $\|v_k\| = \|u_k\|$, going if necessary to a subsequence, there is $v \in H_{\rho}^\alpha(\mathbb{R}^n) \setminus \{0\}$ such that 
$$
v_k \rightharpoonup v\;\;\mbox{in}\;\;H_{\rho}^\alpha(\mathbb{R}^n) \quad \mbox{and} \quad v_k \to v\;\;\mbox{in}\;\;L_{loc}^{p}(\mathbb{R}^n);
$$
Furthermore, by the $\mathbb{Z}^n$ invariance of the problem, $\{v_k\}$ is also a $(PS)_c$ sequence of $I$. Thus for every $\varphi \in C_{0}^{\infty}(\mathbb{R}^n)$, 
$$
I'(v)\varphi  = \lim_{k\to \infty} I'(u_k) \varphi  = 0.
$$ 
So $I'(v) = 0$ and $v$ is a nontrivial weak solution of $(P_\epsilon)$. Moreover, $(f_4)$ together with Fatou's Lemma gives $I(v) \leq c$.  

\vspace{0.5 cm}

\subsection{ Class 2:\,\,\,  $\rho$ is asymptotically periodic}

Hereafter, we denote by $I_\infty:H^{\alpha}_{h_\infty}(\mathbb{R}^n) \to \mathbb{R}$ the functional 
$$
I_\infty(u) = \frac{1}{2}\left( \int_{\mathbb{R}^n}\int_{B(0, h_\infty (x))} \frac{|u(x+z) - u(x)|^2}{|z|^{n+2\alpha}}dzdx + \int_{\mathbb{R}^n}|u(x)|^2dx \right) - \int_{\mathbb{R}^n}F(u(x))dx
$$
and by $w_\infty \in H^{\alpha}_{h_\infty}(\mathbb{R}^n)$ be a nontrivial critical point of $I_\infty$, which was obtained in the last subsection. Then,  
$$
I_\infty(w_\infty) \leq c_\infty \quad \mbox{and} \quad I'_\infty(w_\infty)=0,
$$  
where $c_\infty$ denotes the mountain pass level of $I_\infty$. Since we are assuming $(f_4)$, we know that
$$
c_\infty=\inf_{u \in \mathcal{N}_\infty}I_\infty(u)
$$
where
$$
 \mathcal{N}_\infty=\{u \in H^{\alpha}_{h_\infty}(\mathbb{R}^n) \setminus \{0\} \,:\, I'_{\infty}(u)u=0\},
$$
and so, $I_\infty(w_\infty) = c_\infty$. If $c$ denotes the mountain pass level associated with $I$, the condition $(\rho_3)$ gives  $c \leq c_\infty$. Next, we will study the following situations:
$$
c=c_\infty \quad \mbox{and} \quad c< c_\infty.
$$
\noindent {\bf Case 1:\,\, ${\bf c=c_\infty}$. } \, As $\rho \leq h_\infty$ and $I_\infty'(w_\infty)w_\infty=0$, we have that
$$
I'(w_\infty)w_\infty \leq 0,
$$
hence there is $t \in (0,1]$ such that 
$$
tw_\infty \in \mathcal{N}=\{u \in H^{\alpha}_{\rho}(\mathbb{R}^n) \setminus \{0\} \,:\, I'(u)u=0\}.
$$
By $(f_4)$,
$$
c=\inf_{u \in \mathcal{N}}I(u),
$$
then, as $t \in (0,1]$,
$$
c \leq I(tw_\infty)=I(tw_\infty)-\frac{1}{2}I'(tw_\infty)(tw_\infty)\leq I_\infty(w_\infty)-\frac{1}{2}I_\infty'(w_\infty)(w_\infty),
$$
that is, 
$$
c \leq I_\infty(w_\infty)-\frac{1}{2}I'_\infty(w_\infty)(w_\infty)=I_\infty(w_\infty)=c_\infty.
$$
Since we are supposing that $c=c_\infty$, we deduce that $u^*=tw_\infty$ verifies 
$$
I(u^*)=c \quad \mbox{and} \quad I'(u^*)=0.
$$
By $(f_4)$, it is easy to prove that $u^*$ is a critical for $I$, which finishes the proof.  

\vspace{0.5 cm}

\noindent{\bf Case 2:\,\, ${\bf c<c_\infty.}$}\,\, Hereafter, we denote by $\{u_n\} \subset H^{\alpha}_\rho(\mathbb{R}^n)$ a sequence which satisfies
$$
I(u_k) \to c \quad \mbox{and} \quad I'(u_k) \to 0.
$$ 
By using standard arguments, we know that $\{u_k\}$ is a bounded sequence in $H^{\alpha}_\rho(\mathbb{R}^n)$. Hence, for some subsequence, there is $u \in H^{\alpha}_\rho(\mathbb{R}^n)$ such that
$$
u_k \rightharpoonup u \quad \mbox{in} \quad H^{\alpha}_\rho(\mathbb{R}^n).
$$
{\bf Claim:}\,\, $u \not=0$.  \\
\noindent If $u=0$, there are $R,\eta>0$ and $\{y_k\} \subset \mathbb{R}^n$ such that
\begin{equation} \label{E1}
\limsup_{k \to +\infty}\int_{B_R(y_k)}|u_k|^{2}\,dx \geq \eta.
\end{equation}
Indeed, otherwise we must have 
$$
\lim_{n \to +\infty} \sup_{y \in \mathbb{R}^N}\int_{B_R(y)}|u_k|^{2}\,dx=0.
$$
Then, by Lemma \ref{lemCC}, 
$$
u_k \to 0 \quad \mbox{in} \quad L^{q}(\mathbb{R}^n) \quad \forall q \in (2,2^*), 
$$
from where it follows that
$$
\int_{\mathbb{R}^n} f(u_k)u_k \,dx \to 0.
$$
The above limit together with  $I'(u_k)u_k=o_n(1)$ implies that $u_k \to 0$ in $H^{\alpha}_\rho(\mathbb{R}^n)$, which contradicts the limit $I(u_k) \to c>0$. 

Setting $v_k(x)=u_n(x+y_k)$ and considering $y_k \in \mathbb{Z}^n$, we have that $\{v_k\}$ is bounded in $H^{\alpha}_{h_\infty}(\mathbb{R}^n)$ and there is $v \in H^{\alpha}_{h_\infty}(\mathbb{R}^n)$ such that
$$
v_k \rightharpoonup v \quad \mbox{in} \quad  H^{\alpha}_{h_\infty}(\mathbb{R}^n)
$$
and
$$
\int_{B_R(0)}|v|^{2}\,dx\geq \eta>0
$$
which shows $v \not=0$. 

From (\ref{E1}), it is easy to see that $|y_k| \to +\infty$. Then, by $(\rho_4)$
$$
\rho(x+y_k) \to h_\infty(x) \quad \forall x \in \mathbb{R}^n \quad \mbox{as} \quad k \to +\infty.
$$
The above limit and $I'(u_k)(v(.-y_k))=o_k(1)$ combine to give 
$$
I'_\infty(v)v \leq 0. 
$$
Thus, there is $s \in (0,1]$ such that $su \in \mathcal{N}_\infty$. Consequently,
$$
c_\infty \leq I_\infty(sv)=I_\infty(sv)-\frac{1}{2}I'_\infty(sv)(sv) \leq I_\infty(v)-\frac{1}{2}I'_\infty(v)(v).
$$
Since
$$
I_\infty(v)-\frac{1}{2}I'_\infty(v)(v)=I(v)-\frac{1}{2}I'(v)(v)
$$
it follows 
$$
c_\infty \leq I(v)-\frac{1}{2}I'(v)(v).
$$
On the other hand, the Fatou's Lemma leads to
$$
I(v)-\frac{1}{2}I'(v)(v) \leq \liminf_{k \to +\infty}(I(v_k)-\frac{1}{2}I'(v_k)(v_k))
=\liminf_{k \to +\infty}(I(u_k)-\frac{1}{2}I'(u_k)(u_k))
$$ 
that is,
$$
c_\infty \leq \liminf_{k \to +\infty}I(u_k)=c
$$
which is a contradiction, because we are supposing  $c<c_\infty$. 

From this $u \not= 0$ and $I'(u)=0$, which implies that $I$ has a nontrivial critical point. Moreover, by Fatou's Lemma, it is possible to prove that $I(u)=c$.

\subsection{Class 3: $\rho$ has finite global minimum points} 

Hereafter, we will consider the following energy functional \linebreak $J_{\epsilon}:H^{\alpha}_{\rho}(\mathbb{R}^n) \to \mathbb{R}$ defined
by
\[
J_{\epsilon}(u) =\frac{1}{2}\left( \int_{\mathbb{R}^n}\int_{B(0,\rho_\epsilon (x))} \frac{|u(x+z) - u(x)|^2}{|z|^{n+2\alpha}}dzdx + \int_{\mathbb{R}^n}|u(x)|^2dx \right) - \int_{\mathbb{R}^n}F(u(x))dx.
\]
It is easy to see that $J_{\epsilon}\in C^{1}\left( H^{\alpha}_{\rho}(\mathbb{R}^n),\mathbb{R}\right)  $ with
$$
\begin{aligned}
J_{\epsilon}'(u)v &= \int_{\mathbb{R}^n}\int_{B(0,\rho_\epsilon (x))}\frac{[u(x+z)-u(x)][v(x+z)-v(x)]}{|z|^{n+2\alpha}}dzdx+ \int_{\mathbb{R}^n}u(x)v(x)dx\\
& - \int_{\mathbb{R}^n}f(u(x))v(x)dx,
\end{aligned}
$$
for any $u,v\in  H^{\alpha}_{\rho}(\mathbb{R}^n)$. Thus, the critical points of $ J_{\epsilon} $ are (weak) solutions  of  $(P_\epsilon)$. Since the functional $J_{\epsilon}$ is not bounded from below on $  H^{\alpha}_\rho(\mathbb{R}^n)$ , we will work on \emph{Nehari manifold} $ \mathcal{N}_{\epsilon}$ associated with the functional $J_{\epsilon}$, given by
$$
\mathcal{N}_{\epsilon} = \left\{ u \in H^{\alpha}_{\rho}(\mathbb{R}^n) \setminus\{ 0\}: J'_{\epsilon}(u) u = 0 \right\}
$$
and with the level
\[
c_{\epsilon} = \inf_{u \in \mathcal{N}_{\epsilon}} J_{\epsilon}(u).
\]
It is possible to prove that $c_{\epsilon} $ is the mountain pass level of functional $J_{\epsilon}$, see Willem \cite{MW}.

For $ \rho \equiv 1 $, we consider the problem
$$
(-\Delta)_{1}^{\alpha}u + u = f(u)   \mbox{ in } \mathbb{R}^{n}, \quad 
u \in H^{\alpha}_1(\mathbb{R}^{n}).\nonumber \eqno{({ P_\infty})}
$$
Associated with the problem $({P_\infty})$, we have the energy functional \linebreak $ J_{1}:H^{\alpha}_{ \infty}(\mathbb{R}^n)  \to \mathbb{R}$ given by
\[
J_{ \infty}(u) =\frac{1}{2}\left( \int_{\mathbb{R}^n}\int_{B(0,1)} \frac{|u(x+z) - u(x)|^2}{|z|^{n+2\alpha}}dzdx + \int_{\mathbb{R}^n}|u(x)|^2dx \right) - \int_{\mathbb{R}^n}F(u(x))dx,
\]
the level
$$
c_{ \infty} = \inf_{u \in \mathcal{M}_{ \infty}} J_{ \infty}(u)
$$
and the Nehari manifold
$$
\mathcal{M}_{ \infty} = \left\{ u \in H^{\alpha}_{ \infty}(\mathbb{R}^n) \setminus\{ 0\}: J'_{ \infty}(u)u = 0 \right\}.
$$

For $ \rho \equiv \rho_{\infty} $, we fix the problem
$$
(-\Delta)_{\rho_\infty}^{\alpha}u + u = f(u)   \mbox{ in } \mathbb{R}^{n}, \quad 
u \in H^{\alpha}_{\rho_\infty}(\mathbb{R}^{n})\nonumber, \eqno{(P_{\rho_{\infty}})}
$$
and as above, we denote by $ J_{\rho_\infty}, c_{\rho_{\infty}}$ and $\mathcal{M}_{\rho_{\infty}}$ the energy functional, the mountain pass level and Nehari manifold associated with $ (P_{\rho_{\infty}}) $ respectively.

\vspace{0.5 cm}
The following result concerns the behavior of $J_{\epsilon} $ on $ \mathcal{M}_{\epsilon} $. Once its proof is standard, we omit it

\begin{lemma}
The functional $ J_{\epsilon} $ is bounded from below on $ \mathcal{M}_{\epsilon} $. Moreover, $ J_{\epsilon} $ is coercive on $ \mathcal{N}_{\epsilon} $.
\end{lemma}

As an immediate consequence of the last lemma, we have

\begin{corollary}\label{ltdalem}
Let $ \{u_{k}\} $ be a sequence in $ \mathcal{N}_{\epsilon} $ and  $ J_{\epsilon}(u_{k}) \to c_{\epsilon} $. Then $ \{u_{k}\} $ is bounded in $ H^{\alpha}_{\rho}(\mathbb{R}^n) $.
\end{corollary}



The next theorem is a version of a result compactness on Nehari manifolds  due to Alves \cite{alves05} for regional fractional laplacian. It establishes that problem $(P_\infty)$ has a ground state solution.

\begin{teo}\label{TeoComp}
Let $ \{u_{k}\} \subset \mathcal{M}_{\infty} $ be a sequence with $ J_{\infty}(u_{k}) \to c_{\infty} $. Then,
\begin{description}
\item[I.] $ u_{k} \to u $ in $H^{\alpha}_1(\mathbb{R}^n) $,

or

\item[II.] There is $ \{y_{k}\} \subset \mathbb{R}^{n}$ with $|y_k| \to +\infty$ and $w \in H^{\alpha}_1(\mathbb{R}^n) $ such that $ w_{k} = u_{k}(\cdot + y_{k}) \to w $ in $ H^{\alpha}_1(\mathbb{R}^n) $ and $J_{\infty}(w) = c_{\infty}$.
\end{description}
\end{teo}

\noindent \textbf{Proof.} Similarly to Corollary~\ref{ltdalem}, we can assume that $\{u_k\}$ is a bounded sequence, and so, there is $ u \in H^{\alpha}_1(\mathbb{R}^n) $ and a subsequence of $ \{ u_{k} \}$, still denoted by itself, such that $u_k  \rightharpoonup u $  in $ H^{\alpha}_1(\mathbb{R}^n)$. Applying the Ekeland's variational principle, there is a sequence $ \{w_{k}\} $ in $ \mathcal{M}_{\infty} $ with
\[
w_{k} = u_{k} + o_{k}(1), \quad J_{\infty}(w_{k}) \to c_{\infty}
\]
and
\begin{align} \label{eq1}
J'_{\infty}(w_{k}) - \tau_{k} E'_{\infty}(w_{k}) = o_{k}(1),
\end{align}
where $ (\tau_{k}) \subset \mathbb{R} $ and $ E_{\infty}(w) = J'_{\infty}(w)  w $, for any $ w \in H^{\alpha}_1(\mathbb{R}^n) $.

Since $ \{{ w_{k}}\} \subset \mathcal{M}_{\infty} $, (\ref{eq1}) leads to
\[
\tau_{k} E'_{\infty}(w_{k})  w_{k} = o_{k}(1).
\]
Gathering  $(f_4)$ and Lemma \ref{lemCC},  it is possible to prove that there is $\eta_1>0$ such that 
$$
E_\infty'(u)u \leq -\eta_1, \quad \forall u \in \mathcal{M}_{\infty}.
$$
From this, $ \tau_{k} \to 0 $ as $ k \to \infty $,   
\[
J_{\infty}(u_{k}) \to c_{\infty}  \,\,\, \mbox{and} \,\,\, J'_{\infty}(u_{k}) \to 0.
\]
Consequently, $ u $ is critical point of  $ J_{\infty} $.

Next, we will study the following  possibilities: $ u \neq 0 $ or $ u = 0 $.

\vspace{0.5 cm}

\noindent \textbf{Case 1:} $ u \neq 0 $.

\noindent By Fatou's Lemma , it is easy to check that
\begin{align*}
c_{\infty} \leq & J_{\infty}(u) = J_{\infty}(u) - \frac{1}{\theta} J'_{\infty}(u) u \\
= & \left(\frac{1}{2}-\frac{1}{\theta} \right)\|u\|^{2}+\int_{\mathbb{R}^N}(\frac{1}{\theta}f(u)u-F(u))dx\\
\leq & \liminf_{k \to \infty}\left\{ \left(\frac{1}{2}-\frac{1}{\theta} \right)\|u_k\|^{2}+\int_{\mathbb{R}^N}(\frac{1}{\theta}f(u_k)u_k-F(u_k))dx\right\}\\
\leq & \limsup_{k \to \infty}\left\{ \left(\frac{1}{2}-\frac{1}{\theta} \right)\|u_k\|^{2}+\int_{\mathbb{R}^N}(\frac{1}{\theta}f(u_k)u_k-F(u_k))dx\right\}\\
= & \lim_{k \to \infty} \left\{ J_{\infty}(u_{k}) - \frac{1}{\theta} J'_{\infty}(u_{k}) u_{k} \right\} = \, c_{\infty} .\\
\end{align*}
Hence,
$$
\|u_k\|^2 \to \|u\|^2 \quad \mbox{in} \quad  \mathbb{R},
$$
from where it follows that $ u_{k} \to u $ in $  H^{\alpha}_1(\mathbb{R}^n)$.

\vspace{0.5 cm}

\noindent \textbf{Case 2:} $ u = 0 $.

\vspace{0.5 cm}

In this case, we claim that there are $R, \xi>0$ and $ \{ y_{k} \} \subset \mathbb{R}^{n} $ satisfying
\begin{align}\label{lionsfalse}
\limsup_{k \to \infty} \int_{B_{R}(y_{k})} |u_{k}|^{2}dx \geq \xi.
\end{align}
If the claim is false, we must have
\begin{align*}
\limsup_{k \to \infty} \sup_{y \in \mathbb{R}^{n}} \int_{B_{R}(y)} |u_{k}|^{2}dx = 0.
\end{align*}
Thus, by Lemma \ref{lemCC}, 
$$
u_{ k} \to 0  \mbox{ in }  L^{p}(\mathbb{R}^{n}), \quad \forall p \in (2,2_{\alpha}^{*}).
$$
Recalling $ J'_{\infty}(u_{k}) u_{k} = o_{k}(1) $, the last limit yields
\[
\|u_k\|^{2} \to 0,
\]
or equivalently
$$
u_k \to 0 \,\,\, \mbox{in} \,\,\, H^{\alpha}_{1}(\mathbb{R}^{n}),
$$
leading to $c_{\infty} = 0 $, which is absurd. This way, (\ref{lionsfalse}) is true. 
Setting
$$
w_{k}(x) = u_{k}(x + {y}_{k}),
$$
we have that 
$$
J_{\infty}(w_{k}) = J_{\infty}(u_{k}) \,\,\, \mbox{and} \,\,\, \|J'_{\infty}(w_{k})\| = \|J'_{\infty}(u_{k})\|,
$$
that is, $ \{ w_{k} \} $ is a sequence $(PS)_{c_{\infty}} $ for $ J_{\infty} $.  If $ w \in  H^{\alpha}_1(\mathbb{R}^{n}) $ denotes the weak limit of $ \{ w_{n} \} $, it follows from (\ref{lionsfalse}),
$$
\int_{B_{{R}}(0)} |w|^{2}dx \geq \xi,
$$
and so, $ w \neq 0 $.

By repeating the same argument of the first case for the sequence $ \{w_{k}\} $, we deduce that $ w_{k} \to w $ in $  H^{\alpha}_1(\mathbb{R}^{n}) $, $ w \in \mathcal{M}_{\infty} $  and $ J_{\infty}(w) = c_{\infty} $.  \fim

\subsubsection{Estimates involving the minimax levels}

The main goal of this section is to prove some estimates involving the minimax levels $c_{\epsilon}$ and $c_{\infty}$. First of all, we recall the inequality
$$
J_{\infty}(u) \leq J_{\epsilon}(u) \,\,\,\,\, \forall u \in H^{\alpha}_{\rho}(\mathbb{R}^{n}),
$$
which implies
\[
c_{\infty} \leq c_{\epsilon}, \quad \forall \epsilon>0.
\]

\begin{lemma}\label{c0<cf00}
The minimax levels $c_{\epsilon}$ and $c_{\rho_{\infty}}$ satisfy the inequality \linebreak $c_{\epsilon} < c_{\rho_{\infty}}$. Hence, $c_{\infty} < c_{\rho_{\infty}}$.
\end{lemma}
\noindent \textbf{Proof.}
In a manner analogous to Theorem~\ref{TeoComp}, there is $ U \in H^{\alpha}_{\rho}(\mathbb{R}^{n}) $ such that
\[
J_{\rho_{\infty}}(U) = c_{\rho_{\infty}} \quad \mbox{ and } \quad J'_{\rho_{\infty}}(U) = 0.
\]
In the sequel, let $ t > 0 $ be satisfy $ t U \in \mathcal{M}_{\epsilon} $. Thereby,
$$
c_{\epsilon} \leq  J_{\epsilon}(tU).
$$
Since that by $(\rho_5)$, $\rho_{\infty} > \rho(x)$ for all $x \in \mathbb{R}^{n}$, we derive
$$
c_{\epsilon} <  J_{\rho_{\infty}}(tU) \leq \max_{s \geq 0}J_{\rho_{\infty}}(sU) = J_{\rho_{\infty}}(U) = c_{\rho_{\infty}}.
$$
\fim

\vspace{0.5 cm}

Using the last lemma, we are able to prove that  $ J_{\epsilon} $  verifies the $(PS)_{d}$ condition for some values of $d$.

\begin{lemma}\label{Cond-PS}
The functional $ J_{\epsilon} $ satisfies the $(PS)_{d}$ condition for \linebreak $ d \leq c_{\infty} + \tau $, where $\tau=\frac{1}{2}(c_{\rho_\infty}-c_\infty)>0$.
\end{lemma}
\begin{proof}
Let $ \{v_{k}\} \subset H^{\alpha}_{\rho}(\mathbb{R}^{n}) $ be a $(PS)_{d}$ sequence for functional $ J_{\epsilon} $ with $ d \leq c_{\infty} + \tau$. Similarly to Corollary~\ref{ltdalem}, $ \{v_{k}\} $ is a bounded sequence in $  H^{\alpha}_{\rho}(\mathbb{R}^{n}) $, and so, for some subsequence, still denoted by $ \{v_{k}\} $,
\[
v_{k} \rightharpoonup v \mbox{ in } H^{\alpha}_{\rho}(\mathbb{R}^{n}),
\]
for some $v \in H^{\alpha}_{\rho}(\mathbb{R}^{n}).$ Now, by using standard arguments, it is possible to prove that
\begin{align}
J_{\epsilon} (v_{k}) - J_{\epsilon}(w_{k}) - J_{\epsilon}(v) = o_{k}(1) \label{J-J0-on1}
\end{align}
and
\begin{align}
\Vert J'_{\epsilon} (v_{k}) - J'_{\epsilon} (w_{k}) - J'_{\epsilon}(v) \Vert = o_{k}(1), \label{J'-J0'-on1}
\end{align}
where $ w_{k} = v_{k} - v $. Since $ J'_{\epsilon}(v) = 0 $ and  $ J_{\epsilon}(v) \geq 0 $, from (\ref{J-J0-on1})-(\ref{J'-J0'-on1}), $\{w_{k}\}$ is a $(PS)_{d^{*}}$ sequence for $J_{\epsilon}$ with  $d^*=d - J_{\epsilon}(v) \leq  c_{\infty} +\tau$.

\bigskip

\begin{claim} \label{C2} There is $ R > 0 $ such that
\[
\limsup_{k \to \infty} \sup_{y \in \mathbb{R}^{n}} \int_{B_{R}(y)}|w_{k}|^{2} dx= 0.
\]
\end{claim}
If the claim is true, we have
\[
\int_{\mathbb{R}^n} f(w_{k})w_k \ dx\to 0.
\]
On the other hand, by (\ref{J'-J0'-on1}),  we know that $ J'_{\epsilon}(w_{k}) = o_{k}(1) $, then
$$
\|w_k\|^{2} \to 0
$$
that is, $ w_{k} \to 0 $ in $ H^{\alpha}_{\rho}(\mathbb{R}^{n}) $, and so, $ v_{k} \to v $ in $H^{\alpha}_{\rho}(\mathbb{R}^{n}). $

\bigskip

\noindent\textbf{Proof of Claim \ref{C2}:}
If the claim is not true, for each $ R > 0 $ given, we find $ \xi > 0 $ and $ \{y_k\} \subset \mathbb{R}^{n} $  such that
\[
\limsup_{k \to \infty}\int_{B_{R}(y_{k})} |w_{k}|^{2} \geq \xi > 0.
\]
Using that $w_k \rightharpoonup 0$ in $ H^{\alpha}_{\rho}(\mathbb{R}^{n}) $, it follows that $\{y_k\}$ is an unbounded sequence.  Setting
\[
\tilde{w}_{k} = w_{k}(\cdot + y_{k}),
\]
we have that $ \{\tilde{w}_{k}\} $ is bounded in $H^{\alpha}_{\rho}(\mathbb{R}^{n}$) Thus, there are $ \tilde{w} \in H^{\alpha}_{\rho}(\mathbb{R}^{n}) \setminus\{0\} $ and a subsequence of $ \{\tilde{w}_{n}\} $, still denoted by itself,  such that
$$
\tilde{w}_{k} \rightharpoonup \tilde{w} \in H^{\alpha}_{\rho}(\mathbb{R}^{n} ).
$$
Moreover, since $ J'_{\epsilon}(w_{k}) \phi( \cdot - y_{k}) = o_{k}(1) $ for each $ \phi \in  H^{\alpha}_{\rho}(\mathbb{R}^{n})$, we obtain
$$
\begin{aligned}
 0&= \int_{\mathbb{R}^n}\int_{B(0,\rho_\infty )}\frac{[\tilde{w}(x+z)-\tilde{w}(x)][\phi(x+z)-\phi(x)]}{|z|^{n+2\alpha}} + \int_{\mathbb{R}^n}\tilde{w}(x)\phi(x)dx\\
& - \int_{\mathbb{R}^n}f(\tilde{w}(x))\phi(x)dx,
\end{aligned}
$$
from where it follows that $ \tilde{w} $ is a weak solution of $(P_{\rho_\infty})$. Consequently, after some routine calculations, 
$$
c_{\rho_{\infty}}  \leq J_{\rho_{\infty}}(\tilde{w}) = J_{\rho_{\infty}}(\tilde{w}) - \frac{1}{\theta} J'_{\rho_{\infty}}(\tilde{w})\tilde{w} \leq \liminf_{k \to \infty} \left\{J_{\epsilon}(w_{n}) - \frac{1}{\theta} J'_{\epsilon}(w_{n})  w_{n}\right\} = d^*,
$$
that is, $c_{\rho_{\infty}}  \leq c_{\infty} + \tau$, which is an absurd because $\tau < c_{\rho_{\infty}}  - c_{\infty}$. Therefore, the Claim \ref{C2} is true.
\end{proof}

In what follows, let us fix $ \gamma_{0}, r_{0} > 0 $ such that 
\begin{itemize}
  \item $ \overline{B_{\gamma_{0}}(a_{i})} \cap \overline{B_{\gamma_{0}}(a_{j})} = \emptyset $ for $ i \neq j$ \,\,\, \mbox{and} \,\,\, $i,j \in \{1,...,\ell\}$
  \item $ \bigcup^{\ell}_{i = 1}B_{\gamma_{0}}(a_{i}) \subset B_{r_0}(0) $.

  \item $K_{\frac{\gamma_{0}}{2}} = \bigcup^{\ell}_{i = 1}\overline{B_{\frac{\gamma_{0}}{2}}(a_{i})}$
\end{itemize}
Besides this, we define the function $ Q_{ \epsilon} : H^{\alpha}_\rho(\mathbb{R}^{n}) \to \mathbb{R}^n $ by
\begin{align*}
Q_{\epsilon}(u) = \frac{\int_{\mathbb{R}^n} \chi (\epsilon x)|u|^{2}dx}{\int_{\mathbb{R}^n} |u|^{2}dx},
\end{align*}
where  $ \chi : \mathbb{R}^{n}  \to \mathbb{R}^{n} $ is given by
\[
\chi(x) = \left\{
\begin{array}{ccc}
  x & \mbox{if} & |x| \leq r_{0} \\
  r_{0} \frac{x}{|x|} & \mbox{if} & |x| > r_{0}.
\end{array}
\right.
\]

The next two lemmas will be useful to get important $(PS)$-sequences associated with $ J_{\epsilon} $.

\begin{lemma}\label{lemK}
There are $ \delta_{0} > 0 $ and $ \epsilon_1 >0 $ such that if $ u \in \mathcal{M}_{\epsilon} $ and $ J_{\epsilon}(u) \leq c_{\infty} + \delta_{0} $, then
\[
Q_{\epsilon}(u) \in K_{\frac{\gamma_{0}}{2}} \,\,\,\, \mbox{for} \,\,\, \epsilon \in (0, \epsilon_1).
\]
\end{lemma}
\noindent \textbf{Proof.}
If the lemma does not occur, there must be $ \delta_{k} \to 0 $, $ \epsilon_{k} \to 0 $ and $ u_{k} \in \mathcal{M}_{\epsilon} $ such that
\[
J_{\epsilon_k}(u_{k}) \leq c_{\infty} + \delta_{k}
\]
and
\[
Q_{{ \epsilon_{k}}}(u_{k}) \not\in K_{\frac{\gamma_{0}}{2}}.
\]
Fixing $ s_{k} > 0 $ such that $ s_{k} u_{k} \in \mathcal{M}_{\infty} $, we have
\[
c_{\infty} \leq J_{\infty}(s_{k} u_{k}) \leq J_{\epsilon_{k}} (s_{k}u_{k}) \leq \max_{t \geq 0 } J_{\epsilon_{k}} (tu_{k}) = J_{\epsilon_{k}}(u_{k}) \leq c_{\infty} + \delta_{k}.
\]
Hence,
\[\{s_{k} u_{k}\} \subset \mathcal{M}_{\infty} \,\,\,\, \mbox{and} \,\,\,\, J_{\infty}(s_{k} u_{k}) \to c_{\infty}.
\]

Applying the Ekeland's variational principle, we can assume without loss of generality that
$\{s_{k} u_{k}\} \subset \mathcal{M}_{\infty} $  is a sequence $ (PS)_{c_{\infty}} $ for $ J_{\infty} $, that is,
$$
J_\infty(s_k u_k) \to c_\infty \,\,\,\, \mbox{and} \,\,\,\, J'_{\infty}(s_k u_k) \to 0.
$$
According to Theorem \ref{TeoComp}, we must consider the following cases:
\begin{description}
  \item[i)] $ s_{k}u_{k} \to U \neq 0 $ in $ H^{\alpha}_\rho(\mathbb{R}^{n}) $; \par
\end{description}
or
\begin{description}
\item[ii)] There exists $ \{y_{k}\} \subset \mathbb{Z}^{n} $ with $|y_k| \to +\infty $ such that $ v_{k} = s_{k}u_k(\cdot + y_{k}) $ is convergent in $  H^{\alpha}_\rho(\mathbb{R}^{n}) $ for some $ V \in   H^{\alpha}_\rho(\mathbb{R}^{n}) \setminus \{0\}$.
\end{description}

By a direct computation, we can suppose that $ s_{k} \to s_{0} $ for some $ s_{0} > 0 $. Therefore, without loss of generality,
we can assume that
\begin{equation} \label {LIMIT}
u_{k} \to U  \,\,\, \mbox{or} \,\,\,\, v_{k} = u_k( \,\, \cdot + y_{k}) \to  V \,\,\,\, \mbox{in} \,\,\,   H^{\alpha}_\rho(\mathbb{R}^{n}).
\end{equation}

\bigskip

\noindent\textbf{Analysis of} $\mathbf{i)}$.

\bigskip

By Lebesgue's dominated convergence theorem
\[
Q_{\epsilon_{k}}(u_{k}) = \frac{\int_{\mathbb{R}^n} \chi({\epsilon_{k}} x)|u_{k}|^{2}dx}{\int_{\mathbb{R}^n} |u_{k}|^{2}dx} \to \frac{\int_{\mathbb{R}^n} \chi(0)|U|^{2}dx}{\int_{\mathbb{R}^n} |U|^{2}dx} = 0 \in K_{\frac{\gamma_{0}}{2}},
\]
leading to $ Q_{\epsilon_{k}}(u_{k}) \in K_{\frac{\gamma_{0}}{2}} $ for $k $ large, which is absurd.

\bigskip

\noindent\textbf{Analysis of} $\mathbf{ii)}$. \\

\noindent From the equality $J'_{\epsilon_{k}}(u_{k})(u_k) =0 $, we see that 
$$
\begin{aligned}
0&= \int_{\mathbb{R}^n}\int_{B(0,\rho(\epsilon_k x+\epsilon_ky_k) )}\frac{[{v_k}(x+z)-{v_k}(x)]^{2}}{|z|^{n+2\alpha}}dx + \int_{\mathbb{R}^n}|v_k|^{2}(x)dx\\
& - \int_{\mathbb{R}^n}f(\tilde{w}(x))\phi(x)dx,
\end{aligned}
$$

Now, we will study two cases:
\begin{description}
  \item[I)] $ |\epsilon_{k}y_{k}| \to +\infty $
\end{description}
and
\begin{description}
  \item[II)] $ \epsilon_{k}y_{k} \to y $, for some $y \in \mathbb{R}^{n}$.
\end{description}

If I) holds, the limit (\ref{LIMIT}) gives 
$$
\int_{\mathbb{R}^n}\int_{B(0,\rho_\infty )}\frac{[V(x+z)-V(x)]^{2}}{|z|^{n+2\alpha}}dx + \int_{\mathbb{R}^n}|V|^{2}dx - \int_{\mathbb{R}^n}f(V)(x))V(x)dx=0,
$$
and so, $V \in  \mathcal{M}_\infty$. Thereby, 
$$
c_{\rho_{\infty}} \leq J_{\rho_{\infty}}(V)  = J_{\rho_{\infty}}(V) - \frac{1}{\theta}J'_{\rho_{\infty}}(V)V \leq \liminf_{k \to \infty} \left\{  J_{\infty}(u_{k}) - \frac{1}{\theta}J'_{\infty}(u_{k})u_{k}  \right\} = c_{\infty},
$$
that is, $ c_{\rho_{\infty}} \leq c_{\infty} $,  which contradicts Lemma~\ref{c0<cf00}.

\bigskip



Now, if  $ \epsilon_{k}y_{k} \to y $ for some $y \in \mathbb{R}^{n}$, arguing as above we get 
\begin{align}
c_{\rho(y)} \leq c_{\infty}, \label{eq6}
\end{align}
where $ c_{\rho(y)} $ the mountain pass level of the functional  $ J_{\rho(y)} : H^{\alpha}_\rho(\mathbb{R}^{n}) \to \mathbb{R} $ given by
$$
J_{\rho(y)}(u) = \frac{1}{2}\left( \int_{\mathbb{R}^n}\int_{B(0,\rho(y))} \frac{|u(x+z) - u(x)|^2}{|z|^{n+2\alpha}}dzdx + \int_{\mathbb{R}^n}|u(x)|^2dx \right) - \int_{\mathbb{R}^n}F(u(x))dx.
$$
Observe that
$$
c_{\rho(y)} = \inf_{u \in \mathcal{M}_{\rho(y)}} J_{\rho(y)}(u)
$$
where
$$
\mathcal{M}_{\rho(y)} = \left\{ u \in H^{\alpha}_\rho(\mathbb{R}^{N})\setminus\{ 0\}: J'_{\rho(y)}(u) u = 0 \right\}.
$$
If $ \rho(y) >1 $, a similar argument explored in the proof of Lemma~\ref{c0<cf00} shows that $ c_{\rho(y)} > c_{\infty} $, which contradicts the inequality (\ref{eq6}). Therefore, $ \rho(y) = 1$ and $ y = a_{i} $ for some $ i = 1, \cdots \ell $.
Consequently,
\begin{align*}
Q_{\epsilon_{k}}(u_{k})  &=  \frac{\int_{\mathbb{R}^n} \chi(\epsilon_{k} x)|u_{k}|^{2}dx}{\int_{\mathbb{R}^n} |u_{k}|^{2}dx}\\
 & =  \frac{\int_{\mathbb{R}^n} \chi(\epsilon_{k} x + \epsilon_{k} y_{k})|v_{k}|^{2}dx}{\int_{\mathbb{R}^n} |v_{k}|^{2}dx} \to  \frac{\int_{\mathbb{R}^n} \chi(y)|V|^{2}dx}{\int_{\mathbb{R}^n} |V|^{2}dx}=a_i \in K_{\frac{\gamma_{0}}{2}}. \\
\end{align*}
From this, $ Q_{\epsilon_{k}}(u_{k}) \in K_{\frac{\gamma_{0}}{2}} $ for $ k $ large, which is a contradiction, since by assumption $ Q_{\epsilon_{k}}(u_{k}) \not\in K_{\frac{\gamma_{0}}{2}} $.
\fim

\vspace{0.5 cm}

From now on, we will use the ensuing notation
\begin{itemize}
  \item $ \theta^{i}_{\epsilon} = \left\{u \in \mathcal{M}_{\epsilon} ; |Q_{\epsilon}(u) - a_{i} | < \gamma_{0} \right\}$,
  \item $ \partial\theta^{i}_{\epsilon} = \left\{u \in \mathcal{M}_{\epsilon} ; |Q_{k}(u) - a_{i} | = \gamma_{0} \right\}$,
  \item $ \beta^{i}_{\epsilon} = \displaystyle\inf_{u \in \theta^{i}_{\epsilon}} J_{\epsilon}(u) $
  \end{itemize}
and
\begin{itemize}
 \item  $ \tilde{\beta}^{i}_{\epsilon} = \displaystyle\inf_{u \in \partial\theta^{i}_{\epsilon}} J_{\epsilon}(u) .$
\end{itemize}

The above numbers are very important in our approach, because we will prove that there is a $(PS)$ sequence of $J_{\epsilon}$ associated with each $\theta^{i}_{\epsilon}$ for $i=1,2,...,\ell$. To this end, we need of the following technical result

\begin{lemma} \label{rho} There is $ \epsilon^{*} >0 $ such that
$$
\beta^{i}_{\epsilon}  < c_{\infty} + \tau \,\,\, \mbox{and} \,\,\, \beta^{i}_{\epsilon} < \tilde{\beta}^{i}_{\epsilon},
$$
for all $ \epsilon \in (0, \epsilon^{*}$), where $\tau=\frac{1}{2}(c_{\rho_\infty}-c_\infty)>0.$
\end{lemma}
\noindent \textbf{Proof.} From now on,  $ U \in H^{\alpha}_\rho(\mathbb{R}^{n}) $ is a ground state solution for $J_\infty$, that is,
\[
J_{\infty}(U) = c_{\infty} \quad \mbox{ and } \quad  J'_{\infty}(U) = 0 \,\,\, \mbox{( See Theorem \ref{TeoComp} )}.
\]
For $ 1 \leq i \leq \ell$, we define the function $ \widehat{U}^{i}_{\epsilon} : \mathbb{R}^{N} \to \mathbb{R}$ by
\[
\widehat{U}^{i}_{\epsilon}(x) = U( x - \frac{a_{i}}{\epsilon}).
\]

\begin{claim}\label{ltda-Jl-coo} For all $i \in \{1,...,\ell \}$, we have that
\[
\limsup_{k \to +\infty}(\sup_{t \geq 0 } J_{\epsilon}(t\widehat{U}^{i}_{\epsilon})) \leq c_{\infty}.
\]
\end{claim}
By change of variable gives
$$
\begin{array}{l}
J_{\epsilon}(t\widehat{U}^{i}_{\epsilon}) =\displaystyle \frac{t^{2}}{2}\left( \int_{\mathbb{R}^n}\int_{B(0,\rho(\epsilon x+a_i))} \frac{|U(x+z) - U(x)|^2}{|z|^{n+2\alpha}}dzdx + \int_{\mathbb{R}^n}|U(x)|^2dx \right) - \\
\mbox{}\\ \;\;\;\;\;\;\;\;\;\;\;\;\;\;\;\;\;\; \displaystyle\int_{\mathbb{R}^n}F(tU(x))dx.
\end{array}
$$
Moreover, we know that there exists $ s = s(\epsilon)  > 0 $ such that
\begin{align*}
\max_{t \geq 0} J_{\epsilon}(t\widehat{U}^{i}_{\epsilon}) = J_{\epsilon}(s\widehat{U}^{i}_{\epsilon}).
\end{align*}
By a direct computation, it follows that $ s(\epsilon) \not\to 0 $ and  $ s(\epsilon) \not\to \infty $ as $ \epsilon \to 0 $. Thus, without loss of generality, we can assume $ s(\epsilon) \to s_{0}>0$ as $ \epsilon \to 0 $. Thereby,
$$
\begin{array}{l}
\displaystyle \limsup_{\epsilon \to 0} \left( \max_{t \geq 0}J_{\epsilon}(\widehat{U}^{i}_{\epsilon})  \right) \leq \displaystyle \frac{s_0^{2}}{2}\left( \int_{\mathbb{R}^n}\int_{B(0,\rho(\epsilon x+a_i))} \frac{|U(x+z) - U(x)|^2}{|z|^{n+2\alpha}}dzdx + \int_{\mathbb{R}^n}|U(x)|^2dx \right) - \\
\mbox{}\\ \;\;\;\;\;\;\;\;\;\;\;\;\;\;\;\;\;\; \displaystyle\int_{\mathbb{R}^n}F(s_0U(x))dx.
\end{array}
$$
Consequently,
\begin{align*}
\limsup_{\epsilon \to 0}(\sup_{t \geq 0 } J_{\epsilon}(t\widehat{U}^{i}_{\epsilon})) \leq  c_{\infty} \,\,\, \,\, \mbox{for} \,\,\, i \in \{1,....,\ell\}.
\end{align*}

Since $ Q_{\epsilon}(\widehat{U}^{i}_{k}) \to a_{i} $ as $ \epsilon \to 0 $, then $ \widehat{U}^{i}_{\epsilon} \in \theta^{i}_{\epsilon}  $ for all $ \epsilon $ small  enough. On the other hand, by Claim \ref{ltda-Jl-coo}, $ J_{\epsilon} (\widehat{U}^{i}_{\epsilon}) < c_{\infty} + \frac{\delta_0}{4} $ holds also for $\epsilon$ small enough. This way, there exists $\epsilon^*>0$ such that
$$
\beta^{i}_{\epsilon} < c_{\infty} + \frac{\delta_0}{4}, \,\,\,\, \forall \epsilon \in (0, \epsilon^*).
$$
Thus, decreasing $\delta_0$ if necessary, we can assume that
$$
\beta^{i}_{\epsilon} < c_{\infty} + \tau, \,\,\,\, \forall \epsilon \in (0, \epsilon^*).
$$
In order to prove the other inequality, we observe that Lemma \ref{lemK} yields $ J_{\epsilon}(u) \geq c_{\infty} + \delta_{0} $ for all $ u \in \partial \theta^{i}_{\epsilon} $  and $\epsilon \in (0, \epsilon^*)$. Therefore,
\begin{align*}
\tilde{\beta}^{i}_{\epsilon} \geq c_{\infty} + \frac{\delta_{0}}{2}, \,\,\, \mbox{for} \,\,\, \forall \epsilon \in  (0, \epsilon^*).
\end{align*}
Thereby,
\[
\beta^{i}_{\epsilon} < \tilde{\beta}^{i}_{\epsilon},
\]
for $ \epsilon \in (0, \epsilon^*)$. \fim

\begin{lemma}\label{PSb}
For each $ 1 \leq i \leq\ell$, there exists a $(PS)_{\beta^{i}_{\epsilon}}$ sequence,   $ \left\{ u^{i}_{k} \right\} \subset \theta^{i}_{\epsilon} $ for functional $ J_{\epsilon} $.
\end{lemma}

\noindent \textbf{Proof.} By Lemma~\ref{rho}, we know that $ \beta^{i}_{\epsilon} < \tilde{\beta}^{i}_{\epsilon}$. Then, the lemma follows adapting the same ideas explored in \cite{Lin12}. \fim

\subsubsection{Conclusion of the proof for Class 3.}

Let $ \{ u^{i}_{k} \} \subset \theta^{i}_{\epsilon} $ be a $(PS)_{\beta^{i}_{\epsilon}} $ sequence   for functional $ J_{\epsilon} $ given by Lemma~\ref{PSb}. Since $ \beta^{i}_{\epsilon} < c_{\infty} + \tau$, by Lemma~\ref{Cond-PS} there is $ u^{i}$ such that $ u^{i}_{k} \to u^{i} $ in $ H^{\alpha}_\rho(\mathbb{R}^{n}) $. Thus,
$$
u^{i} \in \theta^{i}_{\epsilon}, \,\,\, J_{\epsilon}(u^i) = \beta^{i}_{\epsilon} \,\,\, \mbox{and} \,\,\, J'_{\epsilon}(u^i) = 0.
$$
Now, we infer that $ u^{i} \neq u^{j} $ for $ i \neq j $ as $ 1 \leq i,j \leq \ell $. To see why, it remains to observe that
$$
Q_k(u^i) \in \overline{B_{\gamma_{0}}(a_{i})} \,\,\, \mbox{and} \,\,\,\, Q_k(u^j) \in \overline{B_{\gamma_{0}}(a_{j})}.
$$
Once 
$$
\overline{B_{\gamma_{0}}(a_{i})} \cap \overline{B_{\gamma_{0}}(a_{j})} = \emptyset \,\,\, \mbox{for} \,\,\, i \not= j,
$$
it follows that $u^i \not= u^j$ for $i \not= j$. From this, $ J_{\epsilon} $ has at least $ \ell $ nontrivial critical points for all $ \epsilon \in (0, \epsilon^*) $, which proves the theorem. \fim

\section{A remark about the existence of Ground state solution}

Now we are going to show that the problem $(P_\epsilon)$ has a ground state by supposing only $(f_1)-(f_3)$ and that $\rho$ belongs to Class 1 or 2. Let
\begin{equation}\label{3Eq15}
m = \inf_{\mathcal{O}} I(u),
\end{equation}
where $\mathcal{O} = \{u\in H^{\alpha}_{\rho}(\mathbb{R}^n) \setminus \{0\}:\;\; I'(u) = 0\}$. 

Suppose that $u$ is an arbitrary critical point of $I$. By { Property \ref{nta1} - (2)}, 
\begin{equation}\label{3Eq16}
\mathcal{F}(t)\geq 0\;\;\mbox{for all}\;\;t\in \mathbb{R}.
\end{equation} 
Then
\begin{equation}\label{3Eq17}
I(u) = I(u) - \frac{1}{2}I'(u)u= \frac{1}{2}\int_{\mathbb{R}^n} \mathcal{F}(u(x))dx \geq 0
\end{equation}
which implies that $m\geq 0$. Therefore $0\leq m \leq I(v) < +\infty$. Let $\{u_k\} \subset \mathcal{O}$ be a sequence such that
$$
I(u_k) \to m\;\;\mbox{as}\;\;k\to \infty.
$$
Then, for some $\beta >0$ we have 
\begin{equation}\label{3Eq18}
\|u_k\| \geq \beta.
\end{equation}
Arguing as in the proof of Lemma \ref{3lem1}, $\{u_k\}$ is bounded in $H^{\alpha}_{\rho}(\mathbb{R}^n)$. Let $\delta$ as in (\ref{3Eq12}) associated to $\{u_k\}$. If $\delta = 0$, then
$$
\lim_{k\to \infty} \int_{\mathbb{R}^n}f(u_k(x))u_k(x)dx = 0,
$$
and hence
\begin{equation}\label{3Eq19}
\|u_k\|^2 = I'(u_k)u_k + \int_{\mathbb{R}^n}f(u_k(x))u_k(x)dx \to 0.
\end{equation}
This contradicts with (\ref{3Eq18}). Therefore $\delta >0$ and there exists a sequence $\{y_k\} \subset \mathbb{Z}^n$ such that $v_k(x) = u_k(x+y_k)$ satisfies
$$
I'(v_k) =0 \quad \mbox{and} \quad I(v_k) = I(u_k) \to m\;\;\mbox{as}\;\;k\to \infty,
$$
and $v_k$ converges weakly to some $v\neq 0$, a nonzero critical point of $I$. Furthermore, by (\ref{3Eq16}) and Fatou's Lemma we deduce
$$
\begin{aligned}
I(v) &= I(v) - \frac{1}{2}I'(v)v = \frac{1}{2}\int_{\mathbb{R}^n}\mathcal{F}(v(x))dx\\
&\leq \liminf_{k\to \infty} \frac{1}{2}\int_{\mathbb{R}^n}\mathcal{F}(v_k(x))dx\\
& = \liminf_{k\to \infty} \left( I(u_k) - \frac{1}{2} I'(u_k)u_k  \right) = m.
\end{aligned}
$$
Therefore, $v$ is a nontrivial critical point of $I$ with $I(v) = m$.

\begin{remark}\label{TMnta}
We note that, by Theorem \ref{main}, $v$ is a nontrivial solution but it is possible that $m = I(v) = 0$, because we are assuming that 
$$
\mathcal{F}(t) \geq 0\;\;\forall t\in \mathbb{R}.
$$
To ensure that  $m>0$, it suffices to assume in addition that 
$$
\mathcal{F}(t) >0,\;\;\mbox{for}\;\;t\neq 0.
$$ 
This is the case if $f$ satisfies the following condition $(f_4)$. In fact 
\begin{equation}\label{3Eq20}
2F(t) = 2\int_{0}^{t} \frac{f(s)}{s}sds < 2\int_{0}^{t} \frac{f(t)}{t}sds = f(t)t,
\end{equation}
which implies that $\mathcal{F}(t) >0$. Furthermore, under this condition we can show that the mountain pass       critical point is a ground state, namely 
$$
m = c=\inf_{u \in \mathcal{N}}I(u),
$$
where 
$$
\mathcal{N} = \{u\in H^{\alpha}_{\rho}(\mathbb{R}^n) \setminus \{0\}:\;\; I'(u)u = 0\}.
$$

In fact, by Remark \ref{mountain}, $I$ has the mountain-pass geometry, and we can introduce the following class of paths:
$$
\Gamma = \{\gamma \in C([0,1], H^{\alpha}_{\rho}(\mathbb{R}^n)):\;\;\gamma (0) = 0,\;\; I(\gamma (1))<0\}.
$$
The mountain-pass level
$$
c = \inf_{\gamma \in \Gamma} \sup_{\sigma\in[0,1]} I(\gamma (\sigma)) >0
$$ 
is therefore associated to $\Gamma$. Furthermore, by Remark \ref{mountain} and following the ideas of \cite{PFCT1}, we can show that for any $u\in H^{\alpha}_{\rho}(\mathbb{R}^n) \setminus \{0\}$, there is a unique $t_u = t(u)>0$ such that $t_u u \in \mathcal{N}$ and
$$
I(t_u u) = \max_{t\geq 0} I(tu),
$$
and we note that 
$$
m_* = \inf_{u\in H^{\alpha}_{\rho}(\mathbb{R}^n) \setminus \{0\}} \max_{t\geq 0} I(tu).
$$
where
$$
m_*=\inf_{u \in \mathcal{N}}I(u)
$$
On the other hand, given any $u\in \mathcal{N}$, we may define the path $\gamma_u(t) = t (t_u u)$, where $T(t_uu)<0$ and obtain that $\gamma_u\in \Gamma$. Thus, $c \leq m_*$.  

The other inequality follows from the fact that, for any $\gamma \in \Gamma$, there exists $t\in (0,1)$ such that $\gamma (1) \in \mathcal{N}$. To prove this fact, we note that if $I'(u)u\geq 0$, then, by (\ref{3Eq20}) we get
$$
\begin{aligned}
I(u) & = \frac{1}{2}\|u\|^{2} - \int_{\mathbb{R}^n}F(u(x))dx\\
& = I'(u)u + \frac{1}{2}\int_{\mathbb{R}^n}\mathcal{F}(u(x))dx\\
&\geq 0  
\end{aligned}
$$  
So, if we assume that $I'(\gamma (t))\gamma (t) >0$ for all $t\in (0,1]$, then $I(\gamma (t))\geq 0$ for all $t \in (0,1]$, contradicting $I(\gamma (1)) <0$. In conclusion, we have proved that
$$
m_* = c.
$$
On the other hand,
$$
m \geq m_* \quad \mbox{and} \quad c \geq m,
$$
from where it follows that
$$
m=m_*=c.
$$

\end{remark}


\begin{thebibliography}{99}

\bibitem {alves05}C.O. Alves,  Existence and Multiplicity of Solution for a Class of Quasilinear Equations. Advanced Nonlinear Studies 5, (2005), 73--87.

\bibitem {ACM} C. O. Alves, P. C. Carri\~ao and O. H. Miyagaki, Nonlinear perturbations of a periodic elliptic problem with critical growth, J. Math. Anal. Appl. 260 (2001), 133-146. 


\bibitem{EDNGPEV}E. Di Nezza, G. Patalluci and E. Valdinoci, \emph{``Hitchhiker's guide to the fractional Sobolev spaces''}, Bull. Sci. math., 2012.


\bibitem{CN} D.M. Cao and E.S. Noussair, \emph{Multiplicity of positive and nodal solutions for nonlinear elliptic problem in $\mathbb{R}^{N}$,} Ann. Inst. Henri Poincar\'e 13(5) (1996), 567--588.

\bibitem {Cao} D.M. Cao and H.S. Zhou, \emph{Multiple positive solutions of nonhomogeneous semilinear elliptic equations in $\mathbb{R}^{N} $,} Proc. Roy. Soc. Edinburgh Sect. 126A (1996), 443--463



\bibitem{MC}M. Cheng, \emph{``Bound state for the fractional Schr\"odinger equation with unbounded potential''}, J. Math. Phys., {\bf 53}, 043507 (2012).





\bibitem{JDMX}J. Dong and M.Xu, \emph{``Some solutions to the space fractional Schr\"odinger equation using momentum representation method''},
J. Math. Phys. {\bf 48}, 072105 (2007).

\bibitem{MFEV}M. Fall and E. Valdinoci, \emph{``Uniqueness and nondegeneracy of positive solutions of $(-\Delta)^{\alpha}u + u = u^{p} $ in $\mathbb{R}^{n}$ when $\alpha$ is close to $1$''}, arXiv:1301.4868v2 [math.AP] 14 Jul 2013.

\bibitem{RFEL}R. Frank and E. Lenzmann, \emph{``Uniqueness of non-linear ground states for fractional Laplacians in $\mathbb{R}$''}, Acta Math., {\bf 210} No 2, 261-318 (2013).

\bibitem{PFAQJT}P. Felmer, A. Quaas and J. Tan, \emph{``Positive solutions of nonlinear Schr\"odinger equation with the fractional laplacian''}, Proceedings of the Royal Society of Edinburgh: Section A Mathematics, {\bf 142}, No 6, 1237-1262 (2012).

\bibitem{PFCT}P. Felmer and C. Torres, \emph{``Radial symmetry of ground states for a regional fractional nonlinear Schr\"odinger equation''}, Communication on pure and applied analysis, 13(6), 2395-2406 (2014).  

\bibitem{PFCT1}P. Felmer and C. Torres, \emph{``Non-linear Schr\"odinger equation with non-local regional diffusion''}, Calc. Var. DOI 10.1007/s00526-014-0778-x

\bibitem{RFELLS}R. Frank, E. Lenzmann and L. Silvestre, \emph{``Uniqueness of radial solutions for the fractional Laplacian''}, to appear in Comm. Pure Appl. Math. 2015, DOI 10.1002/cpa.21591.

\bibitem{guan1} Q-Y. Guan,  \emph{``Integration by Parts Formula for Regional Fractional Laplacian."} Commun. Math. Phys. 266, 289Ð329 (2006).

\bibitem{guan2} Q-Y. Guan and Z.M. Ma,  \emph{``The reflected $\alpha$-symmetric stable processes and regional fractional Laplacian.''} Probab. Theory Relat. Fields 134(4), 649Ð694 (2006)

\bibitem{XGMX}X. Guo and M. Xu, \emph{``Some physical applications of fractional Schr\"odinger equation''}, J. Math. Phys. {\bf 47}, 082104 (2006).

\bibitem{HIGN}H. Ishii and G. Nakamura, \emph{``A class of integral equations and approximation of p-Laplace equations''}, Calc. Var. {\bf 37}, 485-522(2010).

\bibitem{LJ}L. Jeanjean, \emph{``On the existence of bounded Palais-Smale sequences and application to a Landesman-Lazer type problem set on $\mathbb{R}^n$''}, Proc. Roy. Soc. Edinburgh {\bf 129}, 787-809 (1999).

\bibitem{NL-1}N. Laskin, \emph{Fractional quantum mechanics and L\'evy path integrals}, Phys. Lett. A {\bf 268}, 298 - 305 (2000).

\bibitem{NL-2}N. Laskin, \emph{``Fractional Schr\"odinger equation''}. Phys. Rev. E {\bf 66}, 056108 (2002).

\bibitem{SL}S. Liu, \emph{``On ground states of superlinear $p$-Laplacian equations in $\mathbb{R}^n$''}, J. Math. Anal. Appl. {\bf 361}, 48-58 (2010).



\bibitem{Hsu1} T.S Hsu, H.L. Lin and C.C Hu, \emph{Multiple positive solutions of quasilinear elliptic equations in $\mathbb{R}^{N}$,} J.Math. Anal. Appl. 388(2012), 500--512.

\bibitem {Lin12} H.L. Lin, \emph{Multiple positive solutions for semilinear elliptic systems,} J. Math. Anal. Appl. 391 (2012), 107--118.


\bibitem{HuTang} K. Hu and C.L. Tang, \emph{Existence and multiplicity of positive solutions of semilinear elliptic equations in unbounded domains},. J. Differential equations 251(2011), 609--629.


\bibitem{EOFCJV}E. de Oliveira, F. Costa, and J. Vaz, \emph{``The fractional Schr\"odinger equation for delta potentials''}, J. Math. Phys. {\bf 51},
123517 (2010).

\bibitem{YPJLCT}Y. Pu, J. Liu and C. Tang, \emph{``Ground states solutions for non-local regional Schr\"odinger equations''}, Electronic Journal of Differential Equations, {\bf 2015}, No. 223, pp. 1-16 (2015).

\bibitem{PR}P.H. Rabinowitz, \emph{``Minimax method in critical point theory with applications to differential equations''}, CBMS Amer. Math. Soc., No {\bf 65}, 1986.

\bibitem{PR1}P.H. Rabinowitz, \emph{``On a class of nonlinear Schrödinguer equations''}, ZAMP, {\bf 43}, 270-291(1992).

\bibitem{CT}C. Torres, \emph{``Nonlinear Dirichlet problem with non local regional diffusion''}, Fract. Calc. Apple. Anal., Vol {\bf 19}, No 2, 379-393(2016).


\bibitem{MW}M. Willem, \emph{Minimax Theorems}, Birkh\"auser, Boston, Basel, Berlin, 1996.



\end{thebibliography}
\end{document}